\documentclass[sn-mathphys,Numbered]{sn-jnl}

\usepackage{graphicx}%
\usepackage{multirow}%
\usepackage{amsmath,amssymb,amsfonts}%
\usepackage{amsthm}%
\usepackage{mathrsfs}%
\usepackage[title]{appendix}%
\usepackage{xcolor}%
\usepackage{textcomp}%
\usepackage{manyfoot}%
\usepackage{booktabs}%
\usepackage{algorithm}%
\usepackage{algorithmicx}%
\usepackage{algpseudocode}%
\usepackage{listings}
\usepackage[all, cmtip]{xy}%
\usepackage{xcolor}

\newcommand{\Tr}{{\rm Tr}}
\newcommand{\N}{{\rm N}}
\newcommand{\gf}{ {{\mathbb F}} }

\newtheorem{theorem}{Theorem}

\newtheorem{proposition}[theorem]{Proposition}%
\newtheorem{lemma}{Lemma}
\newtheorem{remark}{Remark}%

\newtheorem{definition}{Definition}%
\newtheorem{corollary}{Corollary}

\begin{document}

\title[The compositional inverses of the permutation polynomials of the form $x+\gamma\Tr^{q^n}_q(H(x))$ over $\gf_{q^n}$]{The compositional inverses of the permutation polynomials of the form $x+\gamma\Tr^{q^n}_q(H(x))$ over $\gf_{q^n}$}

\author[1]{\fnm{} \sur{Danyao Wu}}\email{wudanyao111@163.com}
\author*[2]{\fnm{} \sur{Xuan  Pang}}\email{pangxuan202503@163.com}
\author[1]{\fnm{} \sur{Zilong He}}\email{zilonghe@dgut.edu.cn}
\author[2]{\fnm{} \sur{Pingzhi Yuan}}\email{yuanpz@scnu.edu.cn}


\affil[1]{\orgdiv{School of Computer Science and Technology}, \orgname{Dongguan University of Technology}, \orgaddress{
		\city{Dongguan}, \postcode{523808}, 
		\country{China}}}

\affil[2]{\orgdiv{School of Mathematics}, \orgname{South China Normal University}, \orgaddress{
		\city{Guangzhou}, \postcode{510631},
		\country{China}}}

\abstract{
	This paper focuses on computing the compositional inverses of permutation polynomials of the form $x+\gamma \operatorname{Tr}_{q}^{q^{n}}(H(x))$ over finite fields via the local method. We explicitly construct compositional inverses for four families of permutation polynomials of this type over $\mathbb{F}_{q^2}$, another four families over $\mathbb{F}_{q^3}$, and one general family over $\mathbb{F}_{q^n}$. The closed-form inverse expressions derived in this work supplement the theory of trace permutation polynomials.
	}




\maketitle
\vspace{-1cm}
\qquad\,\date{\today}
\section{Introduction}\label{sec1}

 Let  $\gf_q$ be the finite field with $q$ elements,  where $q$ is a prime power, and
let $\gf_q[x]$
be the ring of polynomials in a single indeterminate $x$ over $\gf_q$. A polynomial
$f \in\gf_q[x]$ is called a {\em permutation polynomial} of $\gf_q$ if its
associated polynomial mapping $f: c\mapsto f(c)$ from $\gf_q$ to itself is a bijection. The unique polynomial denoted by $f^{-1}(x)$ over $\gf_q$
such that $f(f^{-1}(x))\equiv f^{-1}(f(x)) \equiv x \pmod{x^q-x}$ is called the compositional inverse of $f(x).$ Furthermore,  $f(x)$ is called  an involution when $f^{-1}(x)=f(x).$

The study of permutation polynomials and their compositional inverses over finite
fields in terms of their coefficients is a classical and difficult subject which
attracts people's interest partially due to their wide applications in coding theory
\cite{ding2013cyclic,ding2014binary,laigle2007permutation},
cryptography \cite{rivest1978method,schwenk1998public}, combinatorial design theory \cite{ding2006family}, and other areas of mathematics and engineering \cite{lidl1997finite,lidl1994introduction}.
In general, discovering new classes of permutation polynomials is challenging, and computing the coefficients of their compositional inverses seems even more so, except for several classical classes such as
monomials, linearized polynomials, Dickson polynomials, which have nice structure.
In recent years, the compositional inverses of several classes of permutation polynomials, in either explicit or implicit forms, have been investigated.    Wang \cite{wang2024survey} has  summarized the current methods used in the study of the  compositional inverses of permutation polynomials, including the experimental method, the power sum method, the matrix method, the group algebra method, the piecewise method, the decomposition method,  the commutative diagram method, and the local method.   We refer readers to  \cite{wang2024survey} for more details.

In 2008, Charpin and Kyureghyan investigated permutation polynomials of the form
$G(x) + \gamma \operatorname{Tr}_2^{2^n}(H(x))$,
establishing six families of such polynomials by leveraging the linear structures of Boolean functions. Notably, the technique of deriving functions of type
$G(x) + \gamma \operatorname{Tr}_2^{2^n}(H(x))$
from
$G(x)$
via addition of certain Boolean functions is recognized as the ``switching method'', originally introduced by Edel and Pott \cite{edel2009new} in the context of APN (Almost Perfect Nonlinear) functions, which are of significant cryptographic importance. Further related results can be found in \cite{blondeau2015perfect,carlet2021boolean}.

Following Charpin and Kyureghyan's initial contribution, researchers have extensively explored permutation polynomials of this structure. For instance, Charpin and Kyureghyan \cite{charpin2009when} extended their findings to fields of odd characteristic in 2009. In 2011, Kyureghyan \cite{kyureghyan2011constructing} introduced a constructive approach for broad families of permutation polynomials of the form
$L(x) + L(\gamma)h(f(x))$,
where
$L$
is an
$\mathbb{F}_q$-linear permutation over
$\mathbb{F}_{q^n}$,
$h: \mathbb{F}_q \to \mathbb{F}_q$,
and
$\gamma \in \mathbb{F}_{q^n}$
serves as a
$b$-linear structure of
$f: \mathbb{F}_{q^n} \to \mathbb{F}_q$. More recent works by Cepak et al. \cite{cepak2017permutations} and Xie et al. \cite{xie2021several} have further contributed examples following this framework.

In 2016, Kyureghyan and Zieve \cite{kyureghyan2016permutation} conducted an exhaustive search for permutation polynomials of the type
$x + \gamma \operatorname{Tr}_{q}^{q^n}(x^k)$,
where
$\operatorname{Tr}_{q}^{q^n}(x) = x + x^q + x^{q^2} + \cdots + x^{q^{n-1}}$
denotes the relative trace from
$\mathbb{F}_{q^n}$
to
$\mathbb{F}_q$,
with parameters
$\gamma \in \mathbb{F}_{q^n}^*$,
odd
$q$,
$n > 1$,
and
$q^n < 5000$.
They subsequently generalized the findings into nine infinite families, leaving five exceptional cases unexplained. Later, Ma and Ge \cite{ma2017note} extended two of these exceptional cases into a new infinite family.

Inspired by Kyureghyan and Zieve's work, Li et al. \cite{li2018permutation} in 2018 classified all permutation polynomials of the form
$cx + \operatorname{Tr}_{q}^{q^n}(x^k)$
for
$q = 2^m$,
$c \in \mathbb{F}_q^*$,
$m > 1$,
and
$mn < 14$,
proposing fifteen new families. However, four examples listed in \cite[Table 1]{li2018permutation} remained without generalization.

In 2019, Zha et al. \cite{zha2019permutation} examined permutation polynomials of the form
$x + \operatorname{Tr}_q^{q^n}(h(x))$
over
$\mathbb{F}_{q^n}$,
presenting several new families that accounted for the remaining exceptional cases of Kyureghyan and Zieve \cite{kyureghyan2016permutation} and one case from Li et al. \cite{li2018permutation}.

More recently, research on permutation polynomials of the form \cite{yuan2024algebraicstructurepermutationalpolynomials}
 \begin{equation}\label{4.1}
 	F(x) = G(x) + \gamma \Tr_q^{q^n}(H(x))
 \end{equation}
has continued to develop. Yuan \cite{yuan2024algebraicstructurepermutationalpolynomials} constructed new classes of permutation polynomials of the form \eqref{4.1}
 by introducing a novel algebraic structure for permutation polynomials over $\gf_{q^n}$. This construction provided an answer to an open problem proposed by Charpin and Kyureghyan in \cite{charpin2009when}, which asked to characterize a class of permutation polynomials of the type (\ref{4.1}) in which $G(x)$ is neither a permutation nor a linearized polynomial. Further results of this type include the work of Jiang et al. \cite{jiang2025new}, who investigated permutation polynomials of the form $x + \gamma \Tr_q^{q^2}(h(x))$ over finite fields of even characteristic, and Jiang et al. \cite{jiang2026new}, who studied permutation polynomials of the form $x + \gamma \Tr_q^{q^3}(h(x))$ over finite fields of even characteristic.

 Despite these developments, the compositional inverses of permutation polynomials of the form $F(x) = x + \gamma \Tr_q^{q^n}(H(x))$ remain largely unexplored. This paper focuses on determining the compositional inverses of certain permutation polynomials studied in \cite{zha2019permutation,jiang2026new}, aiming to contribute to this underexamined aspect of the theory. Specifically, this paper adopts the local method to derive the compositional inverses of nine families of permutation polynomials with the form $F(x) = x + \gamma \Tr_q^{q^n}(H(x))$. Two technical routes are adopted in our derivation. On the one hand, we provide the explicit solution formula to the equation $F(x)=a$ for arbitrary $a \in \mathbb{F}_{q^n}$, based on which the compositional inverse of $F(x)$ is established via the local method. On the other hand, we first construct a polynomial $\varphi(x)$ satisfying $\varphi(x)\circ F(x)=\Tr_q^{q^n}(H(x))$, and then further deduce the compositional inverse of $F(x)$ by virtue of the local method.

The remainder of this paper is structured as follows. 
Section~2 introduces the necessary notation and preliminary results. 
In Section~3, we determine the compositional inverses of four classes of 
permutation polynomials of the form
$
x + \gamma \Tr_q^{q^2}(H(x))$ over $ \gf_{q^2}.
$
Section~4 is devoted to constructing the compositional inverses of four 
classes of permutation polynomials of the form
$
x + \gamma \Tr_q^{q^3}(H(x))$ over $ \gf_{q^3}.
$
Finally, Section~5 presents the compositional inverses of one class of 
permutation polynomials of the form
$
x + \gamma \Tr_q^{q^n}(H(x)) $ over $ \gf_{q^n}.
$

\section{Preliminaries}
In this section, we present some auxiliary results that will be needed in the
sequel.

Let $q$ be a prime power and  $n$ a positive integer. The $trace function $ $\Tr_{q}^{q^n}(\cdot)$ from $\gf_{q^n}$ to $\gf_{q}$ is defined by
$$\Tr_{q}^{q^n}(x)=x+x^{q}+x^{q^2}+\cdots+x^{q^{n-1}}, \, \,  x \in \gf_{q^n},$$
and the norm function $\N^{q^n}_q(x)$ from $\mathbb{F}_{q^n}$ to $\mathbb{F}_q$ is defined by
$$
\N^{q^n}_q(x) =x^{(q^n-1)/(q-1)}.
$$

The following lemma is crucial for our development, and we will invoke it multiple times in what follows.

 \begin{lemma}\label{leff-}\cite[Theorem 2.2]{yuan2024local} (Local method)
 	Let $q$ be a prime power and $f(x)$ be a polynomial over $\gf_q.$ Then
 	$f(x)$ is a permutation polynomial over $\gf_q$ if and only if there exist nonempty
 	subsets $S_i$, $i=1, 2, \cdots, t$ of $\gf_q$ and maps $\psi_i(x): \gf_q \rightarrow S_i$, $i=1, 2, \cdots, t$ such that $\psi_i(x)\circ f(x)=\varphi_i(x),$ $i=1, 2, \cdots, t$
 	and $x=F(\varphi_1(x), \varphi_2(x), \cdots, \varphi_t(x)),$ where $F(x_1, x_2, \cdots, x_t)\in \gf_q[x_1, x_2, \cdots, x_t].$ Moreover, the compositional inverse of $f(x)$ is given by
 	$$f^{-1}(x)=F(\psi_1(x), \psi_2(x), \cdots, \psi_t(x)).$$
 \end{lemma}

The local method is a powerful and effective tool for constructing permutation polynomials over finite fields and computing their compositional inverses.  Originating from number theory and algebraic theory, this method  adopts the local--global principle to characterize polynomials over finite fields. Specifically, it deduces the global permutation property of a polynomial from its local characteristics on subsets of the finite field. The essence of the local method lies in characterizing the inherent functional relationship between the polynomial $f(x)$ and the variable $x$, which is achieved by analyzing the composite mappings $\psi_i\circ f(x)$ for a family of properly chosen mappings $\psi_i(x)$.

We next present two lemmas concerning the compositional inverses of linearized permutation polynomials and then use them to derive an explicit expression for the solution of the linearized equation.

\begin{lemma}\label{binomial}\cite[Theorem 2.1]{tuxanidy2017compositional}
	Let $L_r(x) = x^{q^r}-ax$, where $a\in \gf_{q^m}^*$, and $1\leq r\leq m-1$. Then $L_r(x)$ is a permutation polynomial over $\gf_{q^m}$ if and only if the norm $\N^{q^m}_{q^d}(a)\neq 1,$ where $d=\gcd(m ,r).$ In this case, its inverse on $\gf_{q^m}$ is
	$$L_r^{-1}(x)=\frac{\N^{q^m}_{q^d}(a)}{1-\N^{q^m}_{q^d}(a)}\sum_{i=0}^{m/d-1}a^{-\frac{q^{(i+1)r}-1}{q^r-1}}x^{q^{ir}}.$$
\end{lemma}
Based on Lemma \ref{binomial}, we derive the unique solution for a class of affine $q-$ polynomial over $\gf_{q^m}.$

\begin{lemma}\label{lelinearizedsolution}
	Suppose that the equation $x^{q^r}-ax-b=0$ has a unique solution in $\gf_{q^m}$, where $a, b\in \gf_{q^m}^*$ and $1\leq r\leq m-1$, then the solution to the equation is $$x=\frac{\N^{q^m}_{q^d}(a)}{1-\N^{q^m}_{q^d}(a)}\sum_{i=0}^{m/d-1}a^{-\frac{q^{(i+1)r}-1}{q^r-1}}b^{q^{ir}},$$
	where $d=\gcd(m ,r).$
\end{lemma}

\begin{proof}
	Since the equation $x^{q^r}-ax-b=0$ has a unique solution in $\gf_{q^m},$  the linearized polynomial $x^{q^r}-ax$ is a permutation polynomial over $\gf_{q^m}.$
	
Moreover, solving
$x^{q^r}-ax-b=0$ is essentially equivalent to finding the preimage of $b$ under the mapping $x^{q^r}-ax$; or equivalently, evaluating its inverse at $b.$
Thus, by Lemma \ref{binomial}, the desired conclusion follows.
\end{proof}

 Zheng et al. \cite{zheng2019inverses} studied  the compositional inverse of the linearized polynomial of the form $x^4+bx^2+ax$ over $\gf_{2^m}.$ Its inverse has a close  relation with the sequence $$S_{-1}=0, S_0=1, S_i=b^{2^{i-1}}S_{i-1}+a^{2^{i-1}}S_{i-2},$$ where $1\leq i\leq m$ and $a, b \in \gf_{2^m}^*.$

\begin{lemma}\label{lelinearized421} \cite[Corollary 4]{zheng2019inverses}
	Let $L(x)=x^4+bx^2+ax$, where $a,  b  \in  \gf_{2^m}^*$ and $m>1.$ Then $L(x)$ is a permutation polynomial over $\gf_{2^m}$ if and only if $S_m+aS_{m-2}^2=1.$ Moreover, if $L(x)$ permutes $\gf_{2^m},$ the inverse of $L(x)$ over $\gf_{2^m}$ is given by
	$$L^{-1}(x)=\sum_{i=0}^{m-1}\left(S_{m-2-i}^{2^{i+1}}+a^{1-2^{i+1}}S_i\right)x^{2^i}.$$
	
\end{lemma}
Analogous to Lemma \ref{lelinearizedsolution}, we have the following lemma. We omit the details here.
\begin{lemma}\label{lelinearizedsolution421}
Using the notation established in Lemma \ref{lelinearized421}. If $L(x)$ is a permutation polynomial over $\gf_{2^m}$, then for any $C \in \gf_{2^m}$ the unique solution to $L(x) = C$ is given by $$x=\sum_{i=0}^{m-1}\left(S_{m-2-i}^{2^{i+1}}+a^{1-2^{i+1}}S_i\right)C^{2^i}.$$
	\end{lemma}

\section{The  inverses of the permutation polynomials of the form $x+\gamma \mathrm{Tr}_q^{q^2}(H(x))$ over finite fields}

In this section, we determine the compositional inverses of four classes of permutation polynomials of the form
$$
x + \gamma \Tr_q^{q^2}\big(H(x)\big),
$$
where $\gamma \in \gf_{q^2}$ and $H(x)$ is a specific polynomial over $\gf_{q^2}$.

Feng et al. \cite[Lemma 2]{feng2019further} presented that the polynomial
$$f(x)=x(x^n-d)^{(p-1)/n}$$
is a permutation polynomial over
$\gf_{p^m},$ where  $p$ is an odd prime,  $m, n$
are positive integers with $n \mid (p-1)$ and
$d \in \gf_{p^m}$ satisfies $d^{(p^m-1)/n}\neq1$.  As the compositional inverse of such polynomial $f(x)$ is essential for the subsequent analysis in this section, we begin by deriving it.	If $d=0,$ then $f(x)=x^p,$ and so the compositional inverse of $f(x)$ is trivial.  Hence, we only consider $d\neq0$ while computing the compositional inverse of $f(x).$ We introduce two lemmas at first. For simplicity, let $\sharp T$ denote the cardinality of a finite set $T$.
\begin{lemma} \label{AGW} \cite[Lemma 1.2]{akbary2011constructing} (AGW criterion) Let $A, S$ and $\overline{S}$ be finite sets
	with $\sharp S =\sharp \overline{S}$,
	and let $f(x) : A\longrightarrow A$, $g(x): S\longrightarrow \overline{S}$, $\lambda(x): A\longrightarrow S,$ and $\overline{\lambda}(x):A\longrightarrow \overline{S}$ be maps
	such that $\overline{\lambda}(x)\circ f(x)=g(x)\circ \lambda(x).$
	If both $\lambda(x)$ and $\overline{\lambda}(x)$ are surjective,
	then the following statements are equivalent:\\
	(i) $f(x)$ is bijective (a permutation of $A$); and\\
	(ii) $g(x)$ is bijective from $S$ to $\overline{S}$ and
	$f(x)$ is injective on $\lambda^{-1}(s)$ for each $s \in S.$
\end{lemma}

\begin{lemma}\cite{niu2021finding,yuan2024local}\label{lex^rh(x^r)}
	Let $r, s$ be positive integers and $s\mid (q-1).$ Let $f(x)=x^rh(x^s)$ be a permutation polynomial over $\gf_q$ and $g^{-1}(x)$ be a compositional inverse of $g(x)=x^rh(x)^s$ over $\mu_{(q-1)/s}.$ Suppose that  $a$ and $b$ are two positive integers satisfying $as+br=1.$ Then the compositional inverse of $f(x)$ in $\gf_q[x]$ is given by
	$$f^{-1}(x)=g^{-1}(x^s)^ax^bh\left(g^{-1}(x^s)\right)^{-b}.$$
\end{lemma}


We now study the composition inverse of $f(x)=x(x^n-d)^{(p-1)/n}$ in \cite[Lemma 2]{feng2019further}.
\begin{lemma}\label{len(p-1)/n-}
	Let $q=p^m$, where $p$ is an odd prime and $m$
	is a positive integer. Let  $ n$
	be a positive integer with $n \mid (p-1).$
	For any
	$d \in \gf_{q}^*$ with $d^{(q-1)/n}\neq1$,  let
	$f(x)=x(x^n-d)^{(p-1)/n}$
	be a polynomial over
	$\gf_{q}.$  Then the compositional inverse of $f(x)$ is
	{\scriptsize \begin{align*}
			f^{-1}(x)=&\,
			x^{q-n} \left(1-\N_p^q\left(d(x+d)^{n(q-2)}\right)\right)\Big[\N_p^{q}\left(d(x+d)^{n(q-2)}\right)\\
			&\
			\sum_{i=0}^{m-1}\bigl(d^{q-2}(x+d)^n\bigr)^{\frac{p^{i+1}-1}{p-1}}(x+d)^{n(q-p^{i}-1)}\Big]^{q-2}\\
			&\,
			\Bigg[\left(1-\N_p^q\left(d(x+d)^{n(q-2)}\right)\right)\Big[\N_p^{q}\left(d(x+d)^{n(q-2)}\right)\\
			&\,
			\sum_{i=0}^{m-1}\bigl(d^{q-2}(x+d)^n\bigr)^{\frac{p^{i+1}-1}{p-1}}(x+d)^{n(q-p^{i}-1)}\Big]^{q-2}-d\Bigg]^{\frac{(p-1)(n-1)}{n}}.
	\end{align*}}
	Moreover, for any $a \in \gf_{p^m},$ the unique solution of $f(x)=a$ is
	{\scriptsize \begin{align*}
			x=&\,
			a^{q-n} \left(1-\N_p^q\left(d(a+d)^{n(q-2)}\right)\right)\Big[(\N_p^{q}\left(d(a+d)^{n(q-2)}\right)\\
			&\,
			\sum_{i=0}^{m-1}\bigl(d^{q-2}(a+d)^n\bigr)^{\frac{p^{i+1}-1}{p-1}}(a+d)^{n(q-p^{i}-1)}\Big]^{q-2}\\
			&\,
			\Bigg[\left(1-\N_p^q\left(d(a+d)^{n(q-2)}\right)\right)\Big[\N_p^{q}\left(d(a+d)^{n(q-2)}\right)\\
			&\,
			\sum_{i=0}^{m-1}\bigl(d^{q-2}(a+d)^n\bigr)^{\frac{p^{i+1}-1}{p-1}}(a+d)^{n(q-p^{i}-1)}\Big]^{q-2}-d\Bigg]^{\frac{(p-1)(n-1)}{n}}. \end{align*}}
\end{lemma}
\begin{proof}
	It is  clear that $f(x)=0$ if and only if $x=0.$ Let $A=\gf_{q}^*$, $\lambda(x)=x^n-d,$ $\bar{\lambda}(x)=x^n,$ $S=\lambda(A),$ $\bar{S}=\bar{\lambda}(A),$ and $g(x)=(x+d)x^{p-1}.$
	Building on the result of Feng et al. \cite{feng2019further}, in which $f(x)$ was shown to be a permutation polynomial since $f(x)$, $\lambda(x)$, $\bar{\lambda}(x)$, and $g(x)$ satisfy the AGW criterion (Lemma \ref{AGW}), the compositional inverse of $f(x)$ can be found by first determining the compositional inverse of $g(x)$, as indicated by Lemma \ref{lex^rh(x^r)}.
	
	For any $b \in \bar{\lambda}(A),$ we need to find the unique solution of \begin{equation}\label{eqx^p}
		g(x)=x^p+dx^{p-1}=b,
	\end{equation}
	or equivalently, by putting $x=1/y$ (since $xb\neq0$),
	\begin{equation}\label{eqy^p}
		y^p-\frac{d}{b}y-\frac{1}{b}=0.
	\end{equation}
	Since $n\mid (p-1)$, $b \in \bar{\lambda}(\gf_{p^m}^*)$ and $d$ is not an $n$-th power, the equation $y^p-\frac{d}{b}y=0$  has only the trivial solution over  $\gf_{p^m}$.
	Thus, the linearized polynomial $y^p-\frac{d}{b}y$ is a permutation polynomial over $\gf_{p^m}.$ It follows from Lemma \ref{lelinearizedsolution} that the solution of \eqref{eqy^p} is
	$$y=\frac{\N_p^{q}\left(b^{q-2}d\right)}{1-\N_p^{q}\left(b^{q-2}d\right)}\sum_{i=0}^{m-1}\bigl(bd^{q-2}\bigr)^{\frac{p^{i+1}-1}{p-1}}b^{q-p^{i}-1},$$
	and thus the solution of \eqref{eqx^p} is  $$x=\frac{1-\N_p^{q}\left(b^{q-2}d\right)}{\N_p^{q}\left(b^{q-2}d\right)\sum_{i=0}^{m-1}\bigl(bd^{q-2}\bigr)^{\frac{p^{i+1}-1}{p-1}}b^{q-p^{i}-1}}.$$
	Therefore, the compositional inverse of $g(x)$ is
	$$g^{-1}(x)=\frac{1-\N_p^{q}\left(dx^{q-2}\right)}{\N_p^{q}\left(dx^{q-2}\right)\sum_{i=0}^{m-1}\bigl(d^{q-2}x\bigr)^{\frac{p^{i+1}-1}{p-1}}x^{q-p^{i}-1}}.$$
	By taking $s=n, r=1$ and $h(x)=(x-d)^{(p-1)/n}$ in Lemma \ref{lex^rh(x^r)}, we have $a=1$, $b=1-n$. Consequently, by Lemma \ref{lex^rh(x^r)}, the compositional inverse of
	$f(x)$ over $\gf_{p^m}$ is given by
	$$f^{-1}(x)=x^{q-n}\cdot g^{-1}((x+d)^n)\cdot h\left(g^{-1}((x+d)^n\right)^{n-1}.$$
	Finally, substituting the expression of $g^{-1}(x)$ and $h(x)$ into the above formula yields the desired result directly. We are done.
	
\end{proof}

 First, we investigate the compositional inverses of four classes of permutation polynomials with the simplified form $$x + \gamma \Tr_q^{q^2}\big(H(x)\big),$$ where $\gamma \in \mathbb{F}_{q^2}$ and $H(x)$ denotes a specific polynomial defined over $\mathbb{F}_{q^2}$. For any arbitrary element $c\in \mathbb{F}_{q^2}$, we derive the explicit solution to the functional equation $x + \gamma \Tr_q^{q^2}\big(H(x)\big)=c$, based on which the exact compositional inverses of the targeted permutation polynomials are explicitly presented.

Zha et al. \cite[Theorem 3 ]{zha2019permutation} showed that the polynomial $f(x)=x+\gamma\Tr_{q}^{q^2}(4\gamma x^7+x^{q+3})$ permutes $\gf_{q^2}$, where $q$ is a power of $7$ and $\gamma \in \gf_{q^2}^*$ satisfying $\gamma^{2(q-1)}=1$ and $\gamma^{2(q-1)/3}\neq 1.$ We investigate the compositional inverse of this permutation polynomial in following result.

\begin{theorem}
	Let $q=7^k$ with a positive integer $k$. Let $\gamma \in \gf_{q^2}^*$ with $\gamma^{2(q-1)}=1$ and $\gamma^{2(q-1)/3}\neq 1.$ Let $f(x)=x+\gamma\Tr_{q}^{q^2}(4\gamma x^7+x^{q+3})$ be a polynomial  over $\gf_{q^2}.$
	
	If $\gamma^q = \gamma$,
	then the compositional inverse of $f(x)$ over $\gf_{q^2}$ is
	{\scriptsize   \begin{align*}
			f^{-1}(x)=&\,
			B^{q^2-3} \left(1-\N_7^{q^2}\left(-\gamma^{-1}(B-\gamma^{-1})^{3(q^2-2)}\right)\right)\Big[(\N_7^{q^2}\left(-\gamma^{-1}(B-\gamma^{-1})^{3(q^2-2)}\right)\\
			&\,
			\sum_{i=0}^{2k-1}\bigl(-\gamma(B-\gamma^{-1})^3\bigr)^{\frac{7^{i+1}-1}{6}}(B-\gamma^{-1})^{3(q^2-7^{i}-1)}\Big]^{q^2-2}\\
			&\,
			\Bigg[\left(1-\N_7^{q^2}\left(-\gamma^{-1}(B-\gamma^{-1})^{3(q^2-2)}\right)\right)\Big[\N_7^{q^2}\left(-\gamma^{-1}(B-\gamma^{-1})^{3(q^2-2)}\right)\\
			&\,
			\sum_{i=0}^{2k-1}\left(-\gamma(B-\gamma^{-1})^3\right)^{\frac{7^{i+1}-1}{6}}(B-\gamma^{-1})^{3(q^2-7^{i}-1)}\Big]^{q^2-2}+\gamma^{-1}\Bigg]^{4}+3(x^q-x), \end{align*}}
	where $B=\gamma^{-1}(x^q-x)^4+4\gamma^{-2}(x^q-x)+\gamma^{-2}x.$
	
	If $\gamma^q =-\gamma$, then the compositional inverse of $f(x)$ over $\gf_{q^2}$ is
	{\scriptsize\begin{align*}
			f^{-1}(x)=&\,
			C^{q^2-3} \left(1-\N_7^{q^2}\left(\gamma^{-1}(C+\gamma^{-1})^{3(q^2-2)}\right)\right)\Big[\N_7^{q^2}\left(\gamma^{-1}(C+\gamma^{-1})^{3(q^2-2)}\right)\\
			&\,
			\sum_{i=0}^{2k-1}\bigl(\gamma(C+\gamma^{-1})^3\bigr)^{\frac{7^{i+1}-1}{6}}(C+\gamma^{-1})^{3(q^2-7^{i}-1)}\Big]^{q^2-2}\\
			&\,
			\Bigg[\left(1-\N_7^{q^2}\left(\gamma^{-1}(C+\gamma^{-1})^{3(q^2-2)}\right)\right)\Big[\N_7^{q^2}\left(\gamma^{-1}(C+\gamma^{-1})^{3(q^2-2)}\right)\\
			&\,
			\sum_{i=0}^{2k-1}\bigl(\gamma(C+\gamma^{-1})^3\bigr)^{\frac{7^{i+1}-1}{6}}(C+\gamma^{-1})^{3(q^2-7^{i}-1)}\Big]^{q^2-2}-\gamma^{-1}\Bigg]^{4}+4(x^q+x),
	\end{align*}}
	where $C=-\gamma^{-1}(x^q+x)^4-4\gamma^{-2}(x^q+x)+\gamma^{-2}x.$
\end{theorem}
\begin{proof} By the assumption that $\gamma^{2(q-1)}=1$, we have
	$\gamma^q=\gamma$ or $\gamma^q=-\gamma$. To determine the compositional inverse of $f(x)$,  we need to solve the unique solution to the following equation for any $c \in \gf_{q^2}$:
	\begin{equation}\label{eqzha2019ffathm3c}
		f(x)=x+\gamma\Tr_{q}^{q^2}(4\gamma x^7+x^{q+3})=c.\end{equation}
	Taking the $q$-th powers of the both sides of \eqref{eqzha2019ffathm3c} gives
	\begin{equation}\label{eqzha2019ffathm3c^q}
		x^q + \gamma^q \text{Tr}_q^{q^2} \left( 4\gamma x^7 + x^{q+3} \right) = c^q.
	\end{equation}		
	If $\gamma^q = \gamma$, let $\theta = c^q - c$, from \eqref{eqzha2019ffathm3c} and \eqref{eqzha2019ffathm3c^q} we get
	$$x^q = x + \theta.$$
	Substituting it into \eqref{eqzha2019ffathm3c} leads to
	
	$$x + \gamma \left( 4\gamma x^7 + 4\gamma(x^7 + \theta^7) + (x + \theta)x^3 + x(x + \theta)^3 \right) = c,$$
	or equivalently,
	\begin{equation}\label{eqthmzhathm3y}
		y(y^3+\gamma^{-1})^2=\gamma^{-1}\theta^4+4\gamma^{-2}\theta+\gamma^{-2}c,
	\end{equation}
	where $x=y+3\theta.$ 	
	Zha et al. \cite[Lemma 1]{zha2019permutation} showed that  $y(y^3+\gamma^{-1})^2$ permutes $\gf_{q^2}.$
	Take $n=3$ and $d=-\gamma^{-1}$ in Lemma  \ref{len(p-1)/n-}. Since $q$ is a power of 7 (so $3\mid(7-1)$) and  $\gamma^{2(q-1)/3}\neq 1$, we obtain
	$$\left(-\gamma^{-1}\right)^{(q^2-1)/3}=\left((-\gamma^{-1})^{q+1}\right)^{(q-1)/3}=\left((-\gamma^{-1})^q(-\gamma^{-1})\right)^{(q-1)/3}=\gamma^{-2(q-2)/3}\neq1.$$
	Therefore, it follows from Lemma  \ref{len(p-1)/n-} that the unique solution of \eqref{eqthmzhathm3y} is 
	{\scriptsize\begin{align*}
			y=&\,
			a^{q^2-3} \left(1-\N_7^{q^2}\left(-\gamma^{-1}(a-\gamma^{-1})^{3(q^2-2)}\right)\right)\Big[\N_7^{q^2}\left(-\gamma^{-1}(a-\gamma^{-1})^{3(q^2-2)}\right)\\
			&\,
			\sum_{i=0}^{2k-1}\bigl(-\gamma(a-\gamma^{-1})^3\bigr)^{\frac{7^{i+1}-1}{6}}(a-\gamma^{-1})^{3(q^2-7^{i}-1)}\Big]^{q^2-2}\\
			&\,
			\Bigg[\left(1-\N_7^{q^2}\left(-\gamma^{-1}(a-\gamma^{-1})^{3(q^2-2)}\right)\right)\Big[\N_7^{q^2}\left(-\gamma^{-1}(a-\gamma^{-1})^{3(q^2-2)}\right)\\
			&\,
			\sum_{i=0}^{2k-1}\bigl(-\gamma(a-\gamma^{-1})^3\bigr)^{\frac{7^{i+1}-1}{6}}(a-\gamma^{-1})^{3(q^2-7^{i}-1)}\Big]^{q^2-2}+\gamma^{-1}\Bigg]^{4}, \end{align*}}
	where $a=\gamma^{-1}\theta^4+4\gamma^{-2}\theta+\gamma^{-2}c.$
	Hence, if $\gamma^q = \gamma$, then according to Lemma \ref{leff-}, the compositional inverse of $f(x)$ over $\gf_{q^2}$ is
	{\scriptsize\begin{align*}
			f^{-1}(x)=&\,
			B^{q^2-3} \left(1-\N_7^{q^2}\left(-\gamma^{-1}(B-\gamma^{-1})^{3(q^2-2)}\right)\right)\Big[\N_7^{q^2}\left(-\gamma^{-1}(B-\gamma^{-1})^{3(q^2-2)}\right)\\
			&\,
			\sum_{i=0}^{2k-1}\bigl(-\gamma(B-\gamma^{-1})^3\bigr)^{\frac{7^{i+1}-1}{6}}(B-\gamma^{-1})^{3(q^2-7^{i}-1)}\Big]^{q^2-2}\\
			&\,
			\Bigg[\left(1-\N_7^{q^2}\left(-\gamma^{-1}(B-\gamma^{-1})^{3(q^2-2)}\right)\right)\Big[\N_7^{q^2}\left(-\gamma^{-1}(B-\gamma^{-1})^{3(q^2-2)}\right)\\
			&\,
			\sum_{i=0}^{2k-1}\left(-\gamma(B-\gamma^{-1})^3\right)^{\frac{7^{i+1}-1}{6}}(B-\gamma^{-1})^{3(q^2-7^{i}-1)}\Big]^{q^2-2}+\gamma^{-1}\Bigg]^{4}+3(x^q-x), \end{align*}}
	where $B=\gamma^{-1}(x^q-x)^4+4\gamma^{-2}(x^q-x)+\gamma^{-2}x.$

	If $\gamma^q = -\gamma$, let $\beta = c^q + c$, from \eqref{eqzha2019ffathm3c} and \eqref{eqzha2019ffathm3c^q} we get
	$$x^q = -x + \beta.$$
	Substituting it into \eqref{eqzha2019ffathm3c} leads to
	$$x + \gamma(4\gamma x^7 + 4\gamma(x - \beta)^7 - (x - \beta)x^3 - x(x - \beta)^3) = c,$$
	or equivalently,
	\begin{equation}\label{eqthmzhathm3z}
		z(z^3-\gamma^{-1})^2=-\gamma^{-1}\beta^4-4\gamma^{-2}\beta+\gamma^{-2}c,
	\end{equation}
	where $x=z+4\beta.$
	Zha et al. \cite{zha2019permutation} showed that  $	z(z^3-\gamma^{-1})^2$ permutes $\gf_{q^2}.$
	Take $n=3$ and $d=\gamma^{-1}$ in Lemma \ref{len(p-1)/n-}. It is not hard to verify that
	$$\left(\gamma^{-1}\right)^{(q^2-1)/3}=\left((\gamma^{-1})^{q+1}\right)^{(q-1)/3}=\left(\gamma^{-q}\gamma^{-1}\right)^{(q-1)/3}=\gamma^{-2(q-2)/3}\neq1.$$
	Therefore, it follows from Lemma  \ref{len(p-1)/n-} that the unique solution of \eqref{eqthmzhathm3z} is
	{\scriptsize\begin{align*}
			z=&\,
			b^{q^2-3} \left(1-\N_7^{q^2}\left(\gamma^{-1}(b+\gamma^{-1})^{3(q^2-2)}\right)\right)\Big[\N_7^{q^2}\left(\gamma^{-1}(b+\gamma^{-1})^{3(q^2-2)}\right)\\
			&\,
			\sum_{i=0}^{2k-1}\bigl(\gamma(b+\gamma^{-1})^3\bigr)^{\frac{7^{i+1}-1}{6}}(b+\gamma^{-1})^{3(q^2-7^{i}-1)}\Big]^{q^2-2}\\
			&\,
			\Bigg[\left(1-\N_7^{q^2}\left(\gamma^{-1}(b+\gamma^{-1})^{3(q^2-2)}\right)\right)\Big[\N_7^{q^2}\left(\gamma^{-1}(b+\gamma^{-1})^{3(q^2-2)}\right)\\
			&\,
			\sum_{i=0}^{2k-1}\bigl(\gamma(b+\gamma^{-1})^3\bigr)^{\frac{7^{i+1}-1}{6}}(b+\gamma^{-1})^{3(q^2-7^{i}-1)}\Big]^{q^2-2}-\gamma^{-1}\Bigg]^{4}, \end{align*}}
	where $b=-\gamma^{-1}\beta^4-4\gamma^{-2}\beta+\gamma^{-2}c.$
	Hence, if $\gamma^q =-\gamma$, then according to Lemma \ref{leff-}, the compositional inverse of $f(x)$ over $\gf_{q^2}$ is
	{\scriptsize\begin{align*}
			f^{-1}(x)=&\,
			C^{q^2-3} \left(1-\N_7^{q^2}\left(\gamma^{-1}(C+\gamma^{-1})^{3(q^2-2)}\right)\right)\Big[\N_7^{q^2}\left(\gamma^{-1}(C+\gamma^{-1})^{3(q^2-2)}\right)\\
			&\,
			\sum_{i=0}^{2k-1}\bigl(\gamma(C+\gamma^{-1})^3\bigr)^{\frac{7^{i+1}-1}{6}}(C+\gamma^{-1})^{3(q^2-7^{i}-1)}\Big]^{q^2-2}\\
			&\,
			\Bigg[\left(1-\N_7^{q^2}\left(\gamma^{-1}(C+\gamma^{-1})^{3(q^2-2)}\right)\right)\Big[\N_7^{q^2}\left(\gamma^{-1}(C+\gamma^{-1})^{3(q^2-2)}\right)\\
			&\,
			\sum_{i=0}^{2k-1}\bigl(\gamma(C+\gamma^{-1})^3\bigr)^{\frac{7^{i+1}-1}{6}}(C+\gamma^{-1})^{3(q^2-7^{i}-1)}\Big]^{q^2-2}-\gamma^{-1}\Bigg]^{4}+4(x^q+x), \end{align*}}
	where $C=-\gamma^{-1}(x^q+x)^4-4\gamma^{-2}(x^q+x)+\gamma^{-2}x.$
	We are done.
\end{proof}
Zha et al. \cite[Theorem 5 ]{zha2019permutation} showed that the polynomial $f(x)$ over $\mathbb{F}_{q^2}$ ($q$ is a power of 7) of the form
$f(x) = x + \gamma \operatorname{Tr}_q^{q^2} \left( x^7 + \alpha (4x^5 - x^{q+4} - 2x^{2q+3}) + \beta (2x^3 - x^{q+2}) \right)$  is a permutation, where $ \alpha, \beta, \gamma \in \mathbb{F}_q^* $ satisfy $ \alpha^{(q-1)/2}=-1$, $ \beta = -2\alpha^2 $ and $ \gamma = 3\alpha^{-3} $. We investigate the compositional inverse of this permutation polynomial in following result.

\begin{theorem}
	Let $q=7^k$ and $ \alpha, \beta, \gamma \in \mathbb{F}_q^* $ satisfy $ \alpha^{(q-1)/2} = -1 $, $ \beta = -2\alpha^2 $ and $ \gamma = 3\alpha^{-3} $. Let
	$f(x) = x + \gamma \operatorname{Tr}_q^{q^2} \left( x^7 + \alpha (4x^5 - x^{q+4} - 2x^{2q+3}) + \beta (2x^3 - x^{q+2}) \right)$
	be a  polynomial  over $ \mathbb{F}_{q^2} $. Then
	the compositional inverse of $f(x)$ over $\gf_{q^2}$ is {\scriptsize	 \begin{align*}
			f^{-1}(x)=&\,
			B^{q-2} \left(1-\N_7^q\left(2\alpha(B+2\alpha)^{2(q-2)}\right)\right)\Big[\N_7^{q}\left(2\alpha(B+2\alpha)^{2(q-2)}\right)\\
			&\,
			\sum_{i=0}^{k-1}\bigl(4\alpha^{-1}(B+2\alpha)^2\bigr)^{\frac{7^{i+1}-1}{6}}(B+2\alpha)^{2(q-7^{i}-1)}\Big]^{q-2}\\
			&\,
			\Bigg[\left(1-\N_7^q\left(2\alpha(B+2\alpha)^{2(q-2)}\right)\right)\Big[\N_7^{q}\left(2\alpha(B+2\alpha)^{2(q-2)}\right)\\
			&\,
			\sum_{i=0}^{k-1}\bigl(4\alpha^{-1}(B+2\alpha)^2\bigr)^{\frac{7^{i+1}-1}{6}}(B+2\alpha)^{2(q-7^{i}-1)}\Big]^{q-2}-2\alpha\Bigg]^{3}+3(x^q-x), \end{align*}}	
	where $B=	2\gamma^{-1}(x^q+x).$
\end{theorem}
\begin{proof} For any $ c \in \mathbb{F}_{q^2} $, we consider the solution of the equation
	\begin{equation}\label{eqthmzha2019ffathm5c}
		x + \gamma \operatorname{Tr}_q^{q^2}(x^7 + \alpha(4x^5 - x^{q+4} - 2x^{2q+3}) + \beta(2x^3 - x^{q+2})) = c
	\end{equation}
	over $\mathbb{F}_{q^2}$. Let $\theta = c^q - c.$ Then $\theta^q = -\theta$ and
	$$x^q = x + \theta$$
	from \eqref{eqthmzha2019ffathm5c}. Substituting it into \eqref{eqthmzha2019ffathm5c} gives
	$$x + \gamma\left(2x^7 + \theta^7 + 2\alpha(x^5 + 6\theta x^4 + 6\theta^2 x^3 + 3\theta^3 x^2 + 6\theta^4 x + 2\theta^5)\right.$$
	$$\left. + 2\beta(x^3 + 5\theta x^2 + 6\theta^2 x + \theta^3)\right) = c.$$
	Let $ x = y + 3\theta $, we can deduce that $ y \in \mathbb{F}_q $ and
	
	\begin{equation}\label{eqthmzha2019ffathm5y}
		2\gamma^{-1}\theta + 4\gamma^{-1}c = y^7 + \alpha y^5 + \beta y^3 + 4\gamma^{-1}y = y(y^2 - 2\alpha)^3
	\end{equation}
	since $ \alpha, \beta, \gamma \in \mathbb{F}_q^* $, $ \beta = -2\alpha^2 $ and $ \gamma = 3\alpha^{-3} $.
	Zha et al. \cite{zha2019permutation} showed that  $	y(y^2 - 2\alpha)^3$ permutes $\gf_{q}.$
	Setting $n=2$ and $d=2\alpha$ in Lemma \ref{len(p-1)/n-}, it is easy to see that $(2\alpha)^{(q-1)/2}=\alpha^{(q-1)/2}\neq1$ by assumption.
	Therefore, it follows from Lemma  \ref{len(p-1)/n-} that the unique solution of \eqref{eqthmzha2019ffathm5y} is
	{\scriptsize \begin{align*}
			y=&\,
			a^{q-2} \left(1-\N_7^q\left(2\alpha(a+2\alpha)^{2(q-2)}\right)\right)\Big[\N_7^{q}\left(2\alpha(a+2\alpha)^{2(q-2)}\right)\\
			&\,
			\sum_{i=0}^{k-1}\bigl(4\alpha^{-1}(a+2\alpha)^2\bigr)^{\frac{7^{i+1}-1}{6}}(a+2\alpha)^{2(q-7^{i}-1)}\Big]^{q-2}\\
			&\,
			\Bigg[\left(1-\N_7^q\left(2\alpha(a+2\alpha)^{2(q-2)}\right)\right)\Big[\N_7^{q}\left(2\alpha(a+2\alpha)^{2(q-2)}\right)\\
			&\,
			\sum_{i=0}^{k-1}\bigl(4\alpha^{-1}(a+2\alpha)^2\bigr)^{\frac{7^{i+1}-1}{6}}(a+2\alpha)^{2(q-7^{i}-1)}\Big]^{q-2}-2\alpha\Bigg]^{3}, \end{align*}	}	
	where $a=	2\gamma^{-1}\theta + 4\gamma^{-1}c.$
	Hence, the compositional inverse of $f(x)$ over $\gf_{q^2}$ is
	{\scriptsize	 \begin{align*}
			f^{-1}(x)=&\,
			B^{q-2} \left(1-\N_7^q\left(2\alpha(B+2\alpha)^{2(q-2)}\right)\right)\Big[\N_7^{q}\left(2\alpha(B+2\alpha)^{2(q-2)}\right)\\
			&\,
			\sum_{i=0}^{k-1}\bigl(4\alpha^{-1}(B+2\alpha)^2\bigr)^{\frac{7^{i+1}-1}{6}}(B+2\alpha)^{2(q-7^{i}-1)}\Big]^{q-2}\\
			&\,
			\Bigg[\left(1-\N_7^q\left(2\alpha(B+2\alpha)^{2(q-2)}\right)\right)\Big[\N_7^{q}\left(2\alpha(B+2\alpha)^{2(q-2)}\right)\\
			&\,
			\sum_{i=0}^{k-1}\bigl(4\alpha^{-1}(B+2\alpha)^2\bigr)^{\frac{7^{i+1}-1}{6}}(B+2\alpha)^{2(q-7^{i}-1)}\Big]^{q-2}-2\alpha\Bigg]^{3}+3(x^q-x), \end{align*}}	
	%
	where $B=	2\gamma^{-1}(x^q+x).$
	This completes the proof.
\end{proof}

Zha et al. \cite[Theorem 6 ]{zha2019permutation} showed that the polynomial  $f(x)$ over $\mathbb{F}_{q^2}$ ($q$ is a power of 5) of the form
$f(x) = x + \gamma \operatorname{Tr}_q^{q^2} \left( 3\gamma x^5 + 2x^{q+2} - x^3 \right)$ is a permutation, where  $ \gamma \in \mathbb{F}_q^* $ with $ \gamma^{(q-1)/2} = -1 $. We now investigate the compositional inverse of this permutation polynomial in the following.			
\begin{theorem}
	Let $q=5^k$ with a positive integer $k$ and  $ \gamma \in \mathbb{F}_q^* $ with $ \gamma^{(q-1)/2} = -1 $. Let $
	f(x) = x + \gamma \operatorname{Tr}_q^{q^2} \left( 3\gamma x^5 + 2x^{q+2} - x^3 \right)
	$ be a polynomial over $\gf_{q^2}.$ Then the compositional inverse of $f(x)$ over $\gf_{q^2}$ is
	{\scriptsize \begin{align*}
			f^{-1}(x)=&\,
			B^{q-2} \left(1-\N_5^q\left(-\gamma^{-1}(B-\gamma^{-1})^{2(q-2)}\right)\right)\Big[\N_5^{q}\left(-\gamma^{-1}(B-\gamma^{-1})^{2(q-2)}\right)\\
			&\,
			\sum_{i=0}^{k-1}\bigl(-\gamma(B-\gamma^{-1})^2\bigr)^{\frac{5^{i+1}-1}{4}}(B-\gamma^{-1})^{2(q-5^{i}-1)}\Big]^{q-2}\\
			&\,
			\Bigg[\left(1-\N_5^q\left(-\gamma^{-1}(B-\gamma^{-1})^{2(q-2)}\right)\right)\Big[\N_5^{q}\left(-\gamma^{-1}(B-\gamma^{-1})^{2(q-2)}\right)\\
			&\,
			\sum_{i=0}^{k-1}\bigl(-\gamma(B-\gamma^{-1})^2\bigr)^{\frac{5^{i+1}-1}{4}}(B-\gamma^{-1})^{2(q-5^{i}-1)}\Big]^{q-2}+\gamma^{-1}\Bigg]^{2}, \end{align*}}
	where $B=3\gamma^{-2}(x^q+x).$
\end{theorem}			
\begin{proof}			
	For any $ c \in \mathbb{F}_{q^2} $, we consider the solution of the equation
	\begin{equation}\label{eqzha2019thm6c}
		x + \gamma \operatorname{Tr}_q^{q^2} \left( 3\gamma x^5 + 2x^{q+2} - x^3 \right) = c
	\end{equation}
	over $ \mathbb{F}_{q^2} $. Since $ \gamma \in \mathbb{F}_q^* $, from \eqref{eqzha2019thm6c} we have $ x^q = x + \theta $, where $ \theta = c^q - c $ and $ \theta^q = -\theta $. Equation \eqref{eqzha2019thm6c} can then be rewritten as
	$$
	x + \gamma \left( 3\gamma x^5 + 3\gamma(x + \theta)^5 + 2(x + \theta)x^2 + 2x(x + \theta)^2 - x^3 - (x + \theta)^3 \right) = c.
	$$
	By some simplification, the above equation turns to
	$$
	\gamma^2 x^5 + 2\gamma(x^3 + 4\theta x^2 + 2\theta^2 x + 2\theta^3) + x + 3\gamma^2 \theta^5 = c.
	$$
	Let $ y = x + 3\theta $. It leads to $ y \in \mathbb{F}_q $ and
	\begin{equation}\label{eqthmzha2019ffathm6y}
		3\gamma^{-2}\theta + \gamma^{-2}c = y^5 + 2\gamma^{-1}y^3 + \gamma^{-2}y = y(y^2 + \gamma^{-1})^2.
	\end{equation}
	Zha et al. \cite{zha2019permutation} showed that  $	 y(y^2 + \gamma^{-1})^2$ permutes $\gf_{q}.$ Take $n=2$ and $d=-\gamma^{-1}$ in Lemma  \ref{len(p-1)/n-}. Then, it follows from Lemma  \ref{len(p-1)/n-} that the unique solution of \eqref{eqthmzha2019ffathm6y} is
	{\scriptsize  \begin{align*}
			y=&\,
			a^{q-2} \left(1-\N_5^q\left(-\gamma^{-1}(a-\gamma^{-1})^{2(q-2)}\right)\right)\Big[\N_5^{q}\left(-\gamma^{-1}(a-\gamma^{-1})^{2(q-2)}\right)\\
			&\,
			\sum_{i=0}^{k-1}\bigl(-\gamma(a-\gamma^{-1})^2\bigr)^{\frac{5^{i+1}-1}{4}}(a-\gamma^{-1})^{2(q-5^{i}-1)}\Big]^{q-2}\\
			&\,
			\Bigg[\left(1-\N_5^q\left(-\gamma^{-1}(a-\gamma^{-1})^{2(q-2)}\right)\right)\Big[\N_5^{q}\left(-\gamma^{-1}(a-\gamma^{-1})^{2(q-2)}\right)\\
			&\,
			\sum_{i=0}^{k-1}\bigl(-\gamma(a-\gamma^{-1})^2\bigr)^{\frac{5^{i+1}-1}{4}}(a-\gamma^{-1})^{2(q-5^{i}-1)}\Big]^{q-2}+\gamma^{-1}\Bigg]^{2}. \end{align*}}
	%
	where $a=	3\gamma^{-2}\theta + \gamma^{-2}c.$
	Hence, the compositional inverse of $f(x)$ over $\gf_{q^2}$ is
	{\scriptsize \begin{align*}
			f^{-1}=&\,
			B^{q-2} \left(1-\N_5^q\left(-\gamma^{-1}(B-\gamma^{-1})^{2(q-2)}\right)\right)\Big[\N_5^{q}\left(-\gamma^{-1}(B-\gamma^{-1})^{2(q-2)}\right)\\
			&\,
			\sum_{i=0}^{k-1}\bigl(-\gamma(B-\gamma^{-1})^2\bigr)^{\frac{5^{i+1}-1}{4}}(B-\gamma^{-1})^{2(q-5^{i}-1)}\Big]^{q-2}\\
			&\,
			\Bigg[\left(1-\N_5^q\left(-\gamma^{-1}(B-\gamma^{-1})^{2(q-2)}\right)\right)\Big[\N_5^{q}\left(-\gamma^{-1}(B-\gamma^{-1})^{2(q-2)}\right)\\
			&\,
			\sum_{i=0}^{k-1}\bigl(-\gamma(B-\gamma^{-1})^2\bigr)^{\frac{5^{i+1}-1}{4}}(B-\gamma^{-1})^{2(q-5^{i}-1)}\Big]^{q-2}+\gamma^{-1}\Bigg]^{2}, \end{align*}}
	where $B=3\gamma^{-2}(x^q+x).$
	This is the desired result.
\end{proof}
Zha et al. \cite[Theorem 7 ]{zha2019permutation} showed that the polynomial $f(x)$ over $\mathbb{F}_{q^2}$ ($q$ is a power of $11$) of the form
$f(x) = x + \gamma \operatorname{Tr}_q^{q^2} \left( 6\gamma x^{11} + 2x^{3q+3} - x^{q+5} \right)$ is a permutation, where $\gamma \in \mathbb{F}_q^* $ satisfies $ \gamma^{(q-1)/5} \neq 1 .$ We now investigate the compositional inverse of this permutation polynomial in following result.	
\begin{theorem}
	Let $q=11^k$ with a positive integer $k$, and $\gamma \in \mathbb{F}_q^* $ with $ \gamma^{(q-1)/5} \neq 1 .$ Let $
	f(x) = x + \gamma \operatorname{Tr}_q^{q^2} \left( 6\gamma x^{11} + 2x^{3q+3} - x^{q+5} \right)
	$ be a polynomial over $\gf_{q^2}.$ Then the compositional inverse of $f(x)$ over $\gf_{q^2}$ is
	{\scriptsize \begin{align*}
			f^{-1}(x)=&\,
			B^{q-5} \left(1-\N_{11}^q\left(-\gamma^{-1}(B-\gamma^{-1})^{5(q-2)}\right)\right)\Big[\N_{11}^{q}\left(-\gamma^{-1}(B-\gamma^{-1})^{5(q-2)}\right)\\
			&\,
			\sum_{i=0}^{k-1}\bigl(-\gamma(B-\gamma^{-1})^5\bigr)^{\frac{11^{i+1}-1}{10}}(B-\gamma^{-1})^{5(q-11^{i}-1)}\Big]^{q-2}\\
			&\,
			\Bigg[\left(1-\N_{11}^q\left(-\gamma^{-1}(B-\gamma^{-1})^{5(q-2)}\right)\right)\Big[\N_{11}^{q}\left(-\gamma^{-1}(B-\gamma^{-1})^{5(q-2)}\right)\\
			&\,
			\sum_{i=0}^{k-1}\bigl(-\gamma(B-\gamma^{-1})^5\bigr)^{\frac{11^{i+1}-1}{10}}(B-\gamma^{-1})^{5(q-11^{i}-1)}\Big]^{q-2}+\gamma^{-1}\Bigg]^{8}, \end{align*}}
	where $B=	\gamma^{-2}(x + 6(x^q-x) - \gamma(x^q-x)^6).$
\end{theorem}
\begin{proof}
	For any \( c \in \mathbb{F}_{q^2} \), we consider the solution of the equation
	\begin{equation}\label{eqzha2019thm7c}
		x + \gamma \operatorname{Tr}_q^{q^2} \left( 6\gamma x^{11} + 2x^{3q+3} - x^{q+5} \right) = c
	\end{equation}
	over $\gf_{q^2}$. Since $ \gamma \in \mathbb{F}_q^* $, from \eqref{eqzha2019thm7c} we have $ x^q = x + \theta $, where $ \theta = c^q - c$ and $\theta^q = -\theta$.
	\eqref{eqzha2019thm7c} can be rewritten as
	$$
	x + \gamma(6\gamma x^{11} + 6\gamma(x + \theta)^{11} + 4(x + \theta)^3 x^3 - (x + \theta)x^5 - x(x + \theta)^5) = c.
	$$
	Simplifying the above equation turns to
	$$
	\gamma^2 x^{11} + 2\gamma(x^6 + 3\theta x^5 + \theta^2 x^4 + 8\theta^3 x^3 + 3\theta^4 x^2 + 5\theta^5 x) + x + 6\gamma^2\theta^{11} = c.
	$$
	Let \( y = x + 6\theta \). It leads to \( y \in \mathbb{F}_q \) and
	\begin{equation}\label{eqzha2019thm7y}
		\gamma^{-2}(c + 6\theta - \gamma\theta^6) = y^{11} + 2\gamma^{-1}y^6 + \gamma^{-2}y = y(y^5 + \gamma^{-1})^2.
	\end{equation}	Zha et al. \cite{zha2019permutation} showed that  $	 y(y^5 + \gamma^{-1})^2$ permutes $\gf_{q}.$  Setting $n=5$ and $d=-\gamma^{-1}$ in Lemma  \ref{len(p-1)/n-}, we note that $(-\gamma^{-1})^{(q-1)/5}=\gamma^{(1-q)/5}\neq0$.
	Therefore, by Lemma  \ref{len(p-1)/n-}, the unique solution of \eqref{eqzha2019thm7y} is
	{\scriptsize \begin{align*}
			y=&\,
			a^{q-5} \left(1-\N_{11}^q\left(-\gamma^{-1}(a-\gamma^{-1})^{5(q-2)}\right)\right)\Big[\N_{11}^{q}\left(-\gamma^{-1}(a-\gamma^{-1})^{5(q-2)}\right)\\
			&\,
			\sum_{i=0}^{k-1}\bigl(-\gamma(a-\gamma^{-1})^5\bigr)^{\frac{11^{i+1}-1}{10}}(a-\gamma^{-1})^{5(q-11^{i}-1)}\Big]^{q-2}\\
			&\,
			\Bigg[\left(1-\N_{11}^q\left(-\gamma^{-1}(a-\gamma^{-1})^{5(q-2)}\right)\right)\Big[\N_{11}^{q}\left(-\gamma^{-1}(a-\gamma^{-1})^{5(q-2)}\right)\\
			&\,
			\sum_{i=0}^{k-1}\bigl(-\gamma(a-\gamma^{-1})^5\bigr)^{\frac{11^{i+1}-1}{10}}(a-\gamma^{-1})^{5(q-11^{i}-1)}\Big]^{q-2}+\gamma^{-1}\Bigg]^{8}. \end{align*}}
		where $a=	\gamma^{-2}(c + 6\theta - \gamma\theta^6).$
		Hence, the compositional inverse of $f(x)$ over $\gf_{q^2}$ is
		{\scriptsize \begin{align*}
				f^{-1}(x)=&\,
				B^{q-5} \left(1-\N_{11}^q\left(-\gamma^{-1}(B-\gamma^{-1})^{5(q-2)}\right)\right)\Big[\N_{11}^{q}\left(-\gamma^{-1}(B-\gamma^{-1})^{5(q-2)}\right)\\
				&\,
				\sum_{i=0}^{k-1}\bigl(-\gamma(B-\gamma^{-1})^5\bigr)^{\frac{11^{i+1}-1}{10}}(B-\gamma^{-1})^{5(q-11^{i}-1)}\Big]^{q-2}\\
				&\,
				\Bigg[\left(1-\N_{11}^q\left(-\gamma^{-1}(B-\gamma^{-1})^{5(q-2)}\right)\right)\Big[\N_{11}^{q}\left(-\gamma^{-1}(B-\gamma^{-1})^{5(q-2)}\right)\\
				&\,
				\sum_{i=0}^{k-1}\bigl(-\gamma(B-\gamma^{-1})^5\bigr)^{\frac{11^{i+1}-1}{10}}(B-\gamma^{-1})^{5(q-11^{i}-1)}\Big]^{q-2}+\gamma^{-1}\Bigg]^{8}, \end{align*}
		}
		where $B=	\gamma^{-2}(x + 6(x^q-x) - \gamma(x^q-x)^6).$ This completes the proof.
	\end{proof}	
	\section{The  inverses of the permutation polynomials of the form $x+\gamma \mathrm{Tr}_q^{q^3}(H(x))$ over finite fields}
	
	In this section, we determine the compositional inverses of four classes of permutation polynomials of the form
	\[
	x + \gamma \Tr_q^{q^3}\big(H(x)\big),
	\]
	where $\gamma \in \gf_{q^3}$ and $H(x)$ is a specific polynomial over $\gf_{q^3}$.

	Let $k$ be a positive integer and $q = 3^k$. Let $i$ be an integer with $0 \leq i \leq k-1$, and let $\gamma \in \gf_q^*$ be such that $-\gamma$ is a non-square in $\gf_q$. In this setting,  Zha et al. \cite[Theorem 1]{zha2019permutation} constructed a class of  permutation polynomials over $\gf_{q^3}$ of the form
	$$f(x) = x + \gamma \Tr_q^{q^3}\bigl(x^{(q+2)3^i}\bigr).$$
	
	In what follows, we derive the compositional inverse of such polynomials $f(x)$.  We start by explicitly solving the equation $f(x) = c$ for $x$ in term of $c$. Denoting the solution by $x = g(c)$, we obtain $x = g(f(x))$. By Lemma \ref{leff-}, it follows that the compositional inverse of $f(x)$ is given by $f^{-1}(x) = g(x)$.

	\begin{theorem}\label{th(q+2)3^i}
	For a positive integer  $k,$ let $q = 3^k$ and $i$ be an integer with $0\leq i\leq k-1.$ Let $\gamma \in \gf_q^*$ be such that $-\gamma$ is not a square over $\gf_q$ and  $f(x)=x+\gamma\Tr_q^{q^3}(x^{(q+2)3^i}).$ Then if $i=0,$  the compositional inverse of $f(x)$ over $\gf_{q^3}$  is $$f^{-1}(x)=\left(x-\gamma(x-x^q)^{q^2+2}\right)\left(1+\gamma\left(\Tr_q^{q^3}(x)\right)^{2}\right)^{-1},$$ and if  $i\neq0,$  the compositional inverse of $f(x)$ over $\gf_{q^3}$  is
	\begin{align*}
		f^{-1}(x)=&\,\bigl(-\gamma(x-x^q)^{(q^2+2)3^i}
		+x\bigr)\bigl(1-\left(\Tr_q^{q^3}(x)\right)^{q-1}\bigr)\\+&\,\left(\Tr_q^{q^3}(x)\right)^{q-1}\left(\N_{3^d}^{q}(C)\bigl(1-\N_{3^d}^{q}(C)\bigr)^{q-2}\sum_{j=0}^{k/d-1}C^{-\frac{3^{(j+1)i}-1}{3^i-1}}D^{3^{ij}}+x\right),
	\end{align*}
	where $C=-\gamma^{-1}\left(\Tr_q^{q^3}(x)\right)^{q-2\cdot 3^i-1},$  $D=-x^{3^i}-\left(\Tr_q^{q^3}(x)\right)^{q-2\cdot 3^i-1}\bigl(x-x^q\bigr)^{(q^2+2)3^i}$
	and  $d=\gcd(k ,i).$
	\end{theorem}
	\begin{proof}
	For any $c \in \gf_{q^3},$  we need to solve that the equation
	$$x + \gamma \mathrm{Tr}_q^{q^3} \bigl( x^{(q+2)3^i} \bigr) = c.$$
	Zha et al. \cite{zha2019permutation} have shown that the above equation  can be rewritten as
	$$x + \gamma \left( \bigl( \Tr_q^{q^3}(c) \bigr)^2 x + \bigl( c^q - c \bigr)^{q^2+2} \right)^{3^i} = c. $$
	In particular,  if $i=0,$ then
	\begin{equation}\label{eq(1)(2+q)3i}
		\left(1+\gamma\left(\Tr_q^{q^3}(c)\right)^2\right)x+\gamma(c-c^q)^{q^2+2}=c,
	\end{equation}
	and if $i\neq 0,$ then
	\begin{equation}\label{eqzhathm1c}
		\gamma\left(\Tr_q^{q^3}(c)\right)^{2\cdot 3^i}x^{3^i}+x+\gamma(c-c^q)^{(q^2+2)3^i}=c,
	\end{equation}
	Equivalently,  substituting $x=y+c$ with $y \in \gf_q$ into \eqref{eqzhathm1c} yields
	\begin{equation}\label{eq(2)(2+q)3i}
		\gamma\left(\Tr_q^{q^3}(c)\right)^{2\cdot 3^i}y^{3^i}+y+\gamma\left(\Tr_q^{q^3}(c)\right)^{2\cdot 3^i}c^{3^i}+\gamma(c-c^q)^{(q^2+2)3^i}=0.
	\end{equation}
	
	For $i=0, $ since $-\gamma $ is not a square in $\gf_q,$ we have $1+\gamma\left(\Tr_q^{q^3}(c)\right)^2\neq 0$, and thus the solution of \eqref{eq(1)(2+q)3i} is
	$x=\left(c-\gamma(c-c^q)^{q^2+2}\right)\left(1+\gamma\left(\Tr_q^{q^3}(c)\right)^{2}\right)^{-1}.$
	Therefore, the compositional inverse of $f(x)$ over $\gf_{q^3}$ is 	$$f^{-1}(x)=\left(x-\gamma(x-x^q)^{q^2+2}\right)\left(1+\gamma\left(\Tr_q^{q^3}(x)\right)^{2}\right)^{-1}.$$
	
	Next, we consider the case $i\neq 0.$
	
	If $\Tr_q^{q^3}(c)=0$, then the solution to \eqref{eqzhathm1c} is \begin{equation}\label{eqc=0}
		x=c-\gamma(c-c^q)^{(q^2+2)3^i}. \end{equation}
	
	Now suppose that $\Tr_q^{q^3}(c)\neq0.$ Since  Zha et al. \cite{zha2019permutation} have shown that the linearized polynomial $\gamma\left(\Tr_q^{q^3}(c)\right)^{2\cdot 3^i}y^{3^i}+y$
	has only the trivial solution, it follows that
	$\gamma\left(\Tr_q^{q^3}(c)\right)^{2\cdot 3^i}y^{3^i}+y$ is a permutation polynomial over $\gf_q$. It follows from  Lemma \ref{lelinearizedsolution} that the solution to \eqref{eq(2)(2+q)3i} can be expressed as
	\begin{align*}	y=\N_{3^d}^{q}(A)\bigl(1-\N_{3^d}^{q}(A)\bigr)^{q-2}\sum_{j=0}^{k/d-1}A^{-\frac{3^{(j+1)i}-1}{3^i-1}}B^{3^{ij}},
	\end{align*}
	where $A=-\gamma^{-1}\left(\Tr_q^{q^3}(c)\right)^{q-2\cdot 3^i-1}$, $B=-c^{3^i}-\left(\Tr_q^{q^3}(c)\right)^{q-2\cdot 3^i-1}\bigl(c-c^q\bigr)^{(q^2+2) 3^i}$
	and  $d=\gcd(k ,i).$

	Hence,  the solution of \eqref{eqzhathm1c} is \begin{equation}\label{eqcneq0}
		x=\N_{3^d}^{q}(A)\bigl(1-\N_{3^d}^{q}(A)\bigr)^{q-2}\sum_{j=0}^{k/d-1}A^{-\frac{3^{(j+1)i}-1}{3^i-1}}B^{3^{ij}}+c.
	\end{equation}
	Thus, from \eqref{eqc=0} and\eqref{eqcneq0}, we obtain the compositional inverse of $f(x)$
	as
	\begin{align*}
		f^{-1}(x)=&\,\bigl(-\gamma(x-x^q)^{(q^2+2)3^i}
		+x\bigr)\bigl(1-\left(\Tr_q^{q^3}(x)\right)^{q-1}\bigr)\\+&\,\left(\Tr_q^{q^3}(x)\right)^{q-1}\left(\N_{3^d}^{q}(C)\bigl(1-\N_{3^d}^{q}(C)\bigr)^{q-2}\sum_{j=0}^{k/d-1}C^{-\frac{3^{(j+1)i}-1}{3^i-1}}D^{3^{ij}}+x\right),
	\end{align*}
	where $C=-\gamma^{-1}\left(\Tr_q^{q^3}(x)\right)^{q-2\cdot 3^i-1},$  $D=-x^{3^i}-\left(\Tr_q^{q^3}(x)\right)^{q-2\cdot 3^i-1}\bigl(x-x^q\bigr)^{(q^2+2)3^i}$
	and  $d=\gcd(k ,i).$
	We complete the proof.
	\end{proof}
	
	Let $k$ be a positive integer, $q = 3^k$, and let $i$ be an integer such that $0 \leq i \leq k-1$. Let $\gamma \in \mathbb{F}_q^*$ be an element for which $-\gamma$ is a non-square in $\mathbb{F}_q$. Zha et al. \cite{zha2019permutation} studied the permutation property of polynomials of the form
	$f(x)=x+\gamma\Tr_q^{q^3}(x^{(2q+1)3^i})$
	over $\mathbb{F}_{q^3}$. Employing a method similar to that of Theorem \ref{th(q+2)3^i}, we investigate the compositional inverse of this permutation polynomial.
	
	\begin{theorem}\label{th(2q+1)3^i}
	For a positive integer  $k,$ let $q = 3^k$ and $i$ be an integer with $0\leq i\leq k-1.$ Let $\gamma \in \gf_q^*$ be such that $-\gamma$ is not a square over $\gf_q$ and  $f(x)=x+\gamma\Tr_q^{q^3}(x^{(q+2)3^i})$ be a polynomial over $\gf_{q^3}.$  Then if $i=0,$  the compositional inverse of $f(x)$ over $\gf_{q^3}$  is $$f^{-1}(x)=\left(x-\gamma(x^q-x)^{2q^2+1}\right)\left(1+\gamma\left(\Tr_q^{q^3}(x)\right)^{2}\right)^{-1},$$ and if  $i\neq0,$  the compositional inverse of $f(x)$ over $\gf_{q^3}$  is  \begin{align*}
		f^{-1}(x)=&\,\bigl(-\gamma(x^q-x)^{(2q^2+1)3^i}
		+x\bigr)\bigl(1-\Tr_q^{q^3}(x)^{q-1}\bigr)\\+&\,\Tr_q^{q^3}(x)^{q-1}\left(\N_{3^d}^{q}(C)\bigl(1-\N_{3^d}^{q}(C)\bigr)^{q-2}\sum_{j=0}^{k/d-1}C^{-\frac{3^{(j+1)i}-1}{3^i-1}}D^{3^{ij}}+x\right),
	\end{align*}
	where $C=-\gamma^{-1}\left(\Tr_q^{q^3}(x)\right)^{q-2\cdot 3^i-1},$  $D=-x^{3^i}-\left(\Tr_q^{q^3}(x)\right)^{q-2\cdot 3^i-1}\bigl(x^q-x\bigr)^{(2q^2+1)3^i}$
	and  $d=\gcd(k ,i).$
	\end{theorem}
	\begin{proof}
	For any $c \in \mathbb{F}_{q^3} $, we need to solve that the equation
	$$
	x + \gamma \Tr_q^{q^3} \bigl( x^{(2q+1)3^i} \bigr) = c.
	$$
	Zha et al. \cite{zha2019permutation} have shown that the above equation
	leads to
	$$
	x + \gamma \left( \bigl( \Tr_q^{q^3}(c) \bigr)^2 x + \bigl( c^q - c \bigr)^{2q^2 + 1} \right)^{3^i} = c.
	$$
	We can get the desired result similarly as the proof of Theorem \ref{th(q+2)3^i}.
	\end{proof}
	In the subsequent two theorems, a polynomial $\varphi(x)$ is constructed to fulfill $\varphi(x)\circ \left(x+\operatorname{Tr}_q^{q^3}\bigl(H(x)\bigr)\right)=\operatorname{Tr}_q^{q^3}\bigl(H(x)\bigr)$. With the aid of the local method, the compositional inverse of the polynomial $x+\operatorname{Tr}_q^{q^3}\bigl(H(x)$ is subsequently deduced.
	
		Jiang et al. \cite[Theorem 3.3]{jiang2026new} have shown that the polynomial $f(x)$ over $\gf_{q^3}$ of the form $f(x)=x+\gamma\Tr_q^{q^3}\left(x^{q+1}+x^{2q+2}\right)$ is a permutation polynomial  if and only if $\gamma\in \gf_q$ such that  $\gamma t^3+\gamma t+1=0$ has no solution in $\gf_{q}, $ where $q> 2$  is a power of $2.$  Now we investigate the compositional inverse of such polynomial $f(x)$ over $\gf_{q^3}$ in the following theorem.
	\begin{theorem}
		Let $q=2^m$ with a positive integer $m>1.$  For $\gamma \in \gf_{q}^*,$ if the polynomial $f(x)=x+\gamma\Tr_q^{q^3}\left(x^{q+1}+x^{2q+2}\right)$ permutes $\gf_{q^3},$ then the compositional inverse of $f(x)$ over $\gf_{q^3}$ is $$f^{-1}(x)=x+\sum_{i=0}^{m-1}\left(S_{m-2-i}^{2^{i+1}}+\gamma^{2^{i+1}-1}S_i\right)\left(\Tr_q^{q^3}\left(x^{q+1}+x^{2q+2}\right)\right)^{2^i},$$
		where $S_{-1}, S_0, S_1, \cdots$ is the sequence in $\gf_{q}$ with $S_{-1}=0, S_0=1, S_i=S_{i-1}+\gamma^{-2^{i-1}}S_{i-2}.$
	\end{theorem}
	\begin{proof}
		Jiang et. al \cite{jiang2026new} has shown that if $f(x)$ is a permutation polynomial over $\gf_{q^3}$, then $\gamma\in \gf_q$ such that  $\gamma t^3+\gamma t+1=0$ has no solution in $\gf_{q}.$
		
		For simplicity, put $A=\gamma\Tr_q^{q^3}\left(x^{q+1}+x^{2q+2}\right)$. Since $\gamma \in \gf_q^*,$ we have $A \in \gf_q.$
		And so
		$$f^{q+1}(x)=(x+A)^{q+1}=x^{q+1}+A(x^q+x)+A^2,$$ and
		$$f^{2q+2}(x)=x^{2q+2}+A^2(x^q+x)^2+A^4.$$
		These equalities imply that
		\begin{align*}
			&\,\Tr_q^{q^3}\left(f^{q+1}(x)+f^{2q+2}(x)\right)\\
			=&\,\Tr_q^{q^3}\left(x^{q+1}+A(x^q+x)+A^2+x^{2q+2}+A^2(x^q+x)^2+A^4\right)\\
			=&\,\Tr_q^{q^3}\left(x^{q+1}+x^{2q+2}\right)+A\Tr_q^{q^3}(x^q+x)+A^2\Tr_q^{q^3}(x^q+x)^2+(A^2+A^4)\Tr_q^{q^3}(1)\\
			=&\,A^4+A^2+A/\gamma.
		\end{align*}		
		
		Since the equation $\gamma t^3+\gamma t +1=0$ has no solution in $\gf_q,$  the linearized polynomial $A^4+A^2+A/\gamma$ is a permutation polynomial over $\gf_q$. By Lemma \ref{lelinearizedsolution421}, we thus have $$A=\sum_{i=0}^{m-1}\left(S_{m-2-i}^{2^{i+1}}+\gamma^{2^{i+1}-1}S_i\right)\left(\Tr_q^{q^3}\left(f^{q+1}(x)+f^{2q+2}(x)\right)\right)^{2^i},$$
		where $S_{-1}, S_0, S_1, \cdots S_{m-1}$ is a sequence in $\gf_{q}$ with $S_{-1}=0, S_0=1$ and $ S_i=S_{i-1}+\gamma^{-2^{i-1}}S_{i-2}.$
		
		Taking 			
		$$\psi_{1}(x)=\sum_{i=0}^{m-1}\left(S_{m-2-i}^{2^{i+1}}+\gamma^{2^{i+1}-1}S_i\right)\left(\Tr_q^{q^3}\left(x^{q+1}+x^{2q+2}\right)\right)^{2^i},\quad
		\psi_2(x)=x,$$ $$\varphi_1(x)=\sum_{i=0}^{m-1}\left(S_{m-2-i}^{2^{i+1}}+\gamma^{2^{i+1}-1}S_i\right)\left(\Tr_q^{q^3}\left(f^{q+1}(x)+f^{2q+2}(x)\right)\right)^{2^i}, \quad \varphi_2(x)=f(x),$$
		we have $$\varphi_1(x)+\varphi_2(x)=x.$$ Then it follows from Lemma \ref{leff-} that the compositional inverse of $f(x)$ over $\gf_{q^3}$ is $$f^{-1}(x)=x+\sum_{i=0}^{m-1}\left(S_{m-2-i}^{2^{i+1}}+\gamma^{2^{i+1}-1}S_i\right)\left(\Tr_q^{q^3}\left(x^{q+1}+x^{2q+2}\right)\right)^{2^i}.$$
		We complete the proof.
	\end{proof}

	Jiang et al. investigated the permutation behavior of the polynomial of the form $x+\Tr_q^{q^3}(x^{(q^2+q)/2}+x^{2q+1})$ over $\gf_{q^3}$ in \cite[Theorem 4.4]{jiang2026new},  where $m$ is odd and $q=2^m.$ Now we present the compositional inverse of this polynomial below.
	\begin{theorem}
	For an odd integer $m,$ let $q=2^m.$ Then the compositional inverse of $f(x)=x+\Tr_q^{q^3}(x^{(q^2+q)/2}+x^{2q+1})$ over $\gf_{q^3}$ is
	$$f^{-1}(x)=x+\left(\Tr_q^{q^3}\left(	x\right)^3+\Tr_q^{q^3}\left(	x^{(q^2+q)/2}\right)+	\Tr_q^{q^3}\left(	x^{2q+1}\right)\right)^{(2q-1)/3}+\Tr_q^{q^3}\left(x\right).$$
	\end{theorem}
	\begin{proof}
	For the sake of  simplicity, let $B$ denote $\Tr_q^{q^3}(x^{(q^2+q)/2}+x^{2q+1}).$ Then we have $B \in \gf_q$ and  \begin{equation}\label{eqjiangth4.4}
		f(x)=x+B
	\end{equation}
	Consequently, by \eqref{eqjiangth4.4}, we obtain
	\begin{align*}
		\Tr_q^{q^3}\left(	f(x)^{q^2+q}\right)=&\,	\Tr_q^{q^3}\left((x+B)^{q^2+q}\right)\\=&\,\,\Tr_q^{q^3}\left(x^{q^2+q}+B(x^{q^2}+x^q)+B^2\right)\\
		=&\,\Tr_q^{q^3}\left(x^{q^2+q}\right)+B^2,\\
		\Tr_q^{q^3}\left(	f(x)^{2q+1}\right)=&\,\Tr_q^{q^3}\left((x+B)^{2q+1}\right)\\
		=&\,	\Tr_q^{q^3}\left(x^{2q+1}+Bx^{2q}+B^2x+B^3\right)\\
		=&\,\Tr_q^{q^3}(x^{2q+1})+B\Tr_q^{q^3}(x^2)+B^2\Tr_q^{q^3}(x)+B^3,\\
		B\Tr_q^{q^3}\left(f(x)\right)^2=&\,B\Tr_q^{q^3}(x^2)+B^3,\\
		B^2\Tr_q^{q^3}\left(f(x)\right)=
		&\,B^2\Tr_q^{q^3}\left(x\right)+B^3.
	\end{align*}
	Combining these equalities yields that
	\begin{align*}
		&\,\Tr_q^{q^3}\left(	f(x)^{(q^2+q)/2}\right)+	\Tr_q^{q^3}\left(	f(x)^{2q+1}\right)+	B\Tr_q^{q^3}\left(f(x)\right)^2+	B^2\Tr_q^{q^3}\left(f(x)\right)\\
		=&\, \Tr_q^{q^3}\left(x^{(q^2+q)/2}\right)+B+\Tr_q^{q^3}(x^{2q+1})+B^3\\
		=&\,B^3,	
	\end{align*}
	where the last equality follows from the definition $B=\Tr_q^{q^3}(x^{(q^2+q)/2})+\Tr_q^{q^3}(x^{2q+1})$.  Using the above equality, we further derive that
	\begin{align}\label{cubiceq5}
		&\,	\Tr_q^{q^3}\left(	f(x)\right)^3+\Tr_q^{q^3}\left(	f(x)^{(q^2+q)/2}\right)+	\Tr_q^{q^3}\left(	f(x)^{2q+1}\right)\nonumber\\=&\,\Tr_q^{q^3}\left(	f(x)\right)^3+B\Tr_q^{q^3}\left(f(x)\right)^2+	B^2\Tr_q^{q^3}\left(f(x)\right)+B^3\nonumber\\
		=&\,\left(\Tr_q^{q^3}\left(	f(x)\right)+B\right)^3.		\end{align}

	Since $m$ is odd  and $q=2^m$, we have $\gcd(3,q-1) = 1$.  A simple computation then shows that the integer $\frac{2q - 1}{3}$ is exactly the multiplicative inverse of 3 modulo $q - 1$; indeed,
	$$3\cdot \frac{2q-1}{3}
	=2q-1\equiv 2q-1-2(q-1)=1\pmod{q-1}.$$
	Applying this inverse property to \eqref{cubiceq5}, we have
	\begin{equation*}
		\bigl(\mathrm{Tr}_q^{q^3}\bigl(f(x)\bigr)^3 + \mathrm{Tr}_q^{q^3}\bigl(f(x)^{(q^2+q)/2}\bigr) + \mathrm{Tr}_q^{q^3}\bigl(f(x)^{2q+1}\bigr)\bigr)^{\frac{2q - 1}{3}} = \mathrm{Tr}_q^{q^3}\bigl(f(x)\bigr) + B.
	\end{equation*}
	Rearranging the above equation to solve for $B$, we obtain
	\begin{equation}
		B = \bigl(\mathrm{Tr}_q^{q^3}\bigl(f(x)\bigr)^3 + \mathrm{Tr}_q^{q^3}\bigl(f(x)^{(q^2+q)/2}\bigr) + \mathrm{Tr}_q^{q^3}\bigl(f(x)^{2q+1}\bigr)\bigr)^{\frac{2q - 1}{3}} +\mathrm{Tr}_q^{q^3}\bigl(f(x)\bigr).
	\end{equation}
	
	Taking
	\begin{align*}
		&\,
		\psi_1=	\left(\Tr_q^{q^3}\left(	x\right)^3+\Tr_q^{q^3}\left(	x^{(q^2+q)/2}\right)+	\Tr_q^{q^3}\left(	x^{2q+1}\right)\right)^{(2q-1)/3}+\Tr_q^{q^3}\left(x\right),\\
		&\,
		\varphi_1=\left(\Tr_q^{q^3}\left(	f(x)\right)^3+\Tr_q^{q^3}\left(	f(x)^{(q^2+q)/2}\right)+	\Tr_q^{q^3}\left(	f(x)^{2q+1}\right)\right)^{(2q-1)/3}+\Tr_q^{q^3}\left(f(x)\right),
	\end{align*}
	$\psi_2=x$ and $\varphi_2(x)=f(x)$ in Lemma \ref{leff-}, we get $$\varphi_2(x)+\varphi_1(x)=x.$$
	It follows from Lemma \ref{leff-} that the compositional inverse of $f(x)$ over $\gf_{q^3}$ is
	$$f^{-1}(x)=x+\left(\Tr_q^{q^3}\left(	x\right)^3+\Tr_q^{q^3}\left(	x^{(q^2+q)/2}\right)+	\Tr_q^{q^3}\left(	x^{2q+1}\right)\right)^{(2q-1)/3}+\Tr_q^{q^3}\left(x\right).$$
	This is the desired result.
	\end{proof}

	\section{The compositional inverses of the permutation polynomials of the form $x+\gamma \mathrm{Tr}_q^{q^n}(H(x))$ over finite fields}
In this section, we first revisit and rigorously prove the compositional inverse formula for permutation polynomials of the form
$
f(x)=x+\gamma g(x),
$
where $\gamma\in\mathbb{F}_{q^n}$ is a $b$-linear translator of the function $g:\mathbb{F}_{q^n}\rightarrow\mathbb{F}_q$, by adopting the local method. Building on this fundamental result, we characterize several families of permutation polynomials with the form
$
x+\gamma \Tr_q^{q^n}\bigl(H(x)\bigr),
$
and explicitly derive their corresponding compositional inverses. As corollaries, a variety of scattered results on permutation polynomials reported in existing literature can be recovered as special cases of our main conclusions. Furthermore, we explicitly determine the compositional inverse of the polynomial
$
x+\gamma\Tr_{q}^{q^n}(x^3-x^{q+2})
$
defined over the finite field $\mathbb{F}_{q^n}$.
 To facilitate the subsequent analysis and help readers better understand our derivations, we first recall the definition of a $b$-linear translator.

\begin{definition}
	\cite{kyureghyan2011constructing,akbary2011constructing,charpin2009when}
	For $S\subseteq \gf_q,$ $\alpha, b \in \gf_q$ and a map $\lambda: \gf_q\rightarrow \gf_q,$ $\alpha$ is a $b$-linear translator of $\lambda$ with respect to $S$ if $\lambda(x+u\alpha)=\lambda(x)+ub$ for all $x \in \gf_q$ and $u \in S.$
\end{definition}

Kyureghyan \cite{kyureghyan2011constructing}  investigated the inverse of permutation polynomial of the form  $f(x)=x+\gamma g(x)$ over $\gf_{q^n}$, where  $\gamma\in\mathbb{F}_{q^n}$ is  a $b$-linear translator of $g:\mathbb{F}_{q^n}\rightarrow\mathbb{F}_{q}$ and $b\neq -1$ .  We provide a new proof of their result using the local method in the following.

\begin{proposition}\label{thlineartranslator}
	Let $\gamma\in\mathbb{F}_{q^n}$ be a $b$-linear translator of $g(x):\mathbb{F}_{q^n}\rightarrow\mathbb{F}_{q}$. Then $f(x)=x+\gamma g(x)$  is a  permutation polynomial of $\mathbb{F}_{q^n}$ if and only if $b\neq-1$. Moreover, if $f(x)$ is a permutation polynomials over $\gf_{q^n}$, then its inverse is given by $f^{-1}(x)=x-\frac{\gamma}{b+1}g(x)$.
\end{proposition}
\begin{proof}
	Since  $\gamma\in\mathbb{F}_{q^n}$ is  a $b$-linear translator of $g(x)$, we have  $$g(x)\circ f(x)=g(x+\gamma g(x))=g(x)+bg(x)=(b+1)g(x).$$ If $b=-1$, then $f(x)$ is not a permutation polynomial over $\gf_{q^n}.$ Now we assume that $b\neq-1.$
	Taking $\psi_1(x)=g(x), $ $\varphi_1(x)=g(x)\circ f(x), $  $\psi_2(x)=x,$ and $\varphi_2(x)=x,$ we have $$\varphi_2(x)-\frac{\gamma}{b+1}\varphi_1(x)=x.$$ It follows from Lemma \ref{leff-} that the compositional inverse of $f(x)$ is $$f^{-1}(x)=x-\frac{\gamma}{b+1}g(x).$$ This completes the proof
\end{proof}

Applying Proposition  \ref{thlineartranslator}, we immediately obtain the following corollaries.
\begin{corollary}\label{colinear}
	Let $q$ be an power of $p$ and $n$ be a positive integer with $p\mid n.$  Let $L(x)\in \gf_q[x]$ be an additive polynomial. For any $\gamma\in \gf_{q}$, the polynomial $f(x)=x+\gamma\Tr^{q^n}_q\bigl(L(x)\bigr)$ permutes $\gf_{q^n},$ and the compositional inverse of $f(x)$ is $$f^{-1}(x)=x-\gamma\Tr^{q^n}_q\bigl(L(x)\bigr).$$
\end{corollary}
\begin{proof}
	For any $d\in \gf_q,$ 
	since $L(x)\in \gf_q[x]$ is  an additive polynomial and $p \mid n$, we have
	\begin{equation*}
		\Tr_q^{q^n}\bigl(L(x+d)\bigr)=\Tr_q^{q^n}\bigl(L(x)+L(d)\bigr)=\Tr_q^{q^n}\bigl(L(x)\bigr)+nL(d)=\Tr_q^{q^n}\bigl(L(x)\bigr). 
	\end{equation*} 
	Consequently, $1$ is a $0$-linear translator of $\Tr_q^{q^n}\bigl(L(x)\bigr).$ Hence, 	by Proposition \ref{thlineartranslator}, we get 
	the desired result.
\end{proof}
Similarly, we obtain the following result, and we omit the detailed proof here.
\begin{corollary}
	Let $q$ be a power of $2$ and $n$ be  even. For any $\gamma \in \gf_q$, the polynomial  $$f(x)=x+\gamma \Tr_q^{q^n}(x^{q^i+1})$$ is an involution over $\gf_{q^n}.$
\end{corollary}

\begin{corollary}\label{thm(x^{q^i+p^t}-x^{q^j+p^t})}
	Let $ q = p^k$ for a prime $p$ and a positive integer $k$. Let $n$ be a positive integer, let $ i, j, t$ be non-negative integers, and let $ \gamma $ be an element of  $\mathbb{F}_q$.  Then the polynomial $$f(x)=x+\gamma\Tr_q^{q^n}(x^{q^i+p^t}-x^{q^j+p^t})$$ permutes $\gf_{q^n},$ and  the compositional inverse of $f(x)$ over $\gf_{q^n}$ is given by $$f^{-1}(x)=x-\gamma\Tr_q^{q^n}(x^{q^i+p^t}-x^{q^j+p^t}).$$	
\end{corollary}

\begin{remark}\label{re(x^{q^i+p^t}-x^{q^j+p^t})}
	Setting $i = t = 0$ and $j = 1$ in Corollary~\ref{thm(x^{q^i+p^t}-x^{q^j+p^t})},
	we obtain that the compositional inverse of the permutation polynomial
	$x + \gamma \Tr_{q}^{q^n}(x^2 - x^{q+1})$ (with $\gamma \in \gf_q$)
	from Theorem~10 of \cite{zha2019permutation} is
	$$
	x - \gamma \Tr_{q}^{q^n}(x^2 - x^{q+1}).
	$$
	
	Taking $i = t = 0$ in Corollary~\ref{thm(x^{q^i+p^t}-x^{q^j+p^t})},
	we find that the compositional inverse of the permutation polynomial
	$x + \gamma \Tr_{q}^{q^n}(x^2 - x^{q^j+1})$ given by Theorem~3.5(1) in \cite{li2025permutation}
	is
	$$
	x - \gamma \Tr_{q}^{q^n}(x^2 - x^{q^j+1}).
	$$
	
	Choosing $i = 0, t = 1, p = 2,$ and $j = 1$ in Corollary~\ref{thm(x^{q^i+p^t}-x^{q^j+p^t})},
	our result implies that the permutation polynomial
	$x + \gamma \Tr_{q}^{q^2}(x^3 + x^{q+2})$ over $\gf_{q^2}$ from Theorem~3.1 of \cite{jiang2025new}
	is an \emph{involution}; that is, its compositional inverse is itself.
	
	Finally, with $i = 0, t = 1, p = 3,$ and $j = 1$ in Corollary~\ref{thm(x^{q^i+p^t}-x^{q^j+p^t})},
	we conclude that the compositional inverse of the permutation polynomial
	$x + \gamma \Tr_{q}^{q^2}(x^4 - x^{q+3})$ over $\gf_{q^2}$ (where $\gamma \in \gf_q$)
	given in Theorem~3.7 of \cite{li2025permutation} is
	$$
	x - \gamma \Tr_{q}^{q^2}(x^4 - x^{q+3}).
	$$	\end{remark}

Given that $\Tr_{q}^{q^n}(x)$ is a linearized polynomial, the following conclusion is obtained by combining Corollary \ref{colinear} with Corollary \ref{thm(x^{q^i+p^t}-x^{q^j+p^t})}. We omit the details here.

\begin{corollary}\label{colinearbino}
	Let $q$ be a  power of $p$ and $n$ be a positive integer with $p \mid n.$   Let $L(x)\in \gf_q[x]$ be an additive polynomial. For any $a\in \gf_q,$ $$f(x)=x+\Tr_q^{q^n}\bigl(L(x)+a(x^{q^i+p^t}-x^{q^j+p^t})\bigr)$$ is a permutation polynomial over $\gf_{q^n}$ and the compositional inverse of $$f^{-1}(x)=x-\Tr_q^{q^n}\bigl(L(x)+a(x^{q^i+p^t}-x^{q^j+p^t})\bigr).$$
\end{corollary}
\begin{remark}
	Jiang et al. \cite{jiang2025new} investigated the permutation properties of four polynomials over $\gf_{q^2}$ (where $q$ is a power of $2$) defined by:
	\begin{align*}
		f_1(x) &= x + \gamma \Tr_q^{q^2}\bigl(x^3 + x^{q+2}\bigr), \\
		f_2(x) &= x + \gamma \Tr_q^{q^2}\bigl(x + x^2 + x^3 + x^{q+2}\bigr), \\
		f_3(x) &= x + \gamma \Tr_q^{q^2}\bigl(x^2 + x^3 + x^{q+2}\bigr), \\
		f_4(x) &= x + \gamma \Tr_q^{q^2}\bigl(x + x^3 + x^{q+2}\bigr),
	\end{align*}
	where $\gamma \in \gf_{q^2}$. According to Remark \ref{re(x^{q^i+p^t}-x^{q^j+p^t})} and Corollary \ref{colinearbino}, for any $\gamma \in \gf_q$, the polynomials $f_1(x)$, $f_2(x)$, $f_3(x)$ and $f_4(x)$ are all involutions. Thus Corollary \ref{colinearbino} provides a family of explicit involutory permutation polynomials over finite fields of characteristic 2.
\end{remark}

\medskip

Zha et al. \cite[Theorem 9]{zha2019permutation} proposed a class of permutation polynomials over $\gf_{q^n}$ of the form
$$
f(x)=x+\gamma\Tr_{q}^{q^n}\bigl(x^3-x^{q+2}\bigr),
$$
where $q$ is a prime power, $n$ is a positive integer, and $\gamma \in \gf_q^*$ satisfies
$2\gamma\Tr_{q}^{q^n}\bigl(c^{q+1}-c^2\bigr) \neq 1$
for any $c \in \gf_{q^n}$ with $\Tr_{q}^{q^n}\bigl(c^{q+2}-c^3\bigr)=0$. In what follows, we determine the compositional inverse of  this polynomial $f(x)$.
\begin{theorem}
	Let $n, k$ be positive integers and $q=p^k$. Assume $\gamma \in \gf_q^*$ satisfies $2\gamma\Tr_{q}^{q^n}(c^{q+1}-c^2)\neq 1$ for any $c \in \gf_{q^n}$ with $\Tr_{q}^{q^n}(c^{q+2}-c^3)=0.$	 Then the compositional inverse of $f(x)=x+\gamma\Tr_{q}^{q^n}(x^3-x^{q+2})$ over $\gf_{q^n}$ is given by
	$$f^{-1}(x)=x-\frac{\gamma\Tr_{q}^{q^n}(x^3-x^{q+2})}{1-2\gamma\Tr_{q}^{q^n}(x^{q+1}-x^{2})}.$$
\end{theorem}
\begin{proof}
	Since $\gamma \in \gf_q
	$, we have
	\begin{align*}
		&f^3(x)-f^{q+2}(x)+2\gamma f^{q+1}(x)\Tr_{q}^{q^n}(x^3-x^{q+2})-2\gamma f^2(x)\Tr_{q}^{q^n}(x^3-x^{q+2})\\
		=&\ \left(f^3(x)-2\gamma f^2(x)\Tr_{q}^{q^n}(x^3-x^{q+2})\right)-\left(f^{q+2}(x)-2\gamma f^{q+1}(x)\Tr_{q}^{q^n}(x^3-x^{q+2})\right)\\
		=&\,f^2(x)\left(x-\gamma\Tr_{q}^{q^n}(x^3-x^{q+2})\right)-f^{q+1}(x)\left(x-\gamma\Tr_{q}^{q^n}(x^3-x^{q+2})\right)\\
		=&\,\left(x^2-x^{q+1}+(x-x^q)\gamma\Tr_{q}^{q^n}(x^3-x^{q+2})\right)\left(x-\gamma\Tr_{q}^{q^n}(x^3-x^{q+2})\right)\\		=&\,(x^3-x^{q+2})+(x-x^q)\gamma^2\Tr_{q}^{q^n}(x^3-x^{q+2})^2. 	\end{align*}
	Taking the trace on both sides yields $$\Tr_{q}^{q^n}\left(f^3(x)-f^{q+2}(x)\right)+2\gamma\Tr_{q}^{q^n}\left(f^{q+1}(x)-f^2(x)\right)\Tr_{q}^{q^n}(x^3-x^{q+2})=\Tr_{q}^{q^n}(x^3-x^{q+2}),$$
	which is equivalent to	
	$$\Tr_{q}^{q^n}\left(f^3(x)-f^{q+2}(x)\right)=\left(1-2\gamma\Tr_{q}^{q^n}\left(f^{q+1}(x)-f^2(x)\right)\right)\Tr_{q}^{q^n}(x^3-x^{q+2}).$$
	
	If the right-hand side above is non-zero, then $1-2\gamma\Tr_{q}^{q^n}\left(f^{q+1}(x)-f^2(x)\right)\neq0$. We thus obtain
	\begin{equation}\label{th2eq}
		\Tr_{q}^{q^n}(x^3-x^{q+2})=\frac{\Tr_{q}^{q^n}\left(f^3(x)-f^{q+2}(x)\right)}{1-2\gamma\Tr_{q}^{q^n}\left(f^{q+1}(x)-f^2(x)\right)}.\end{equation}
	
	If the right-hand side is zero,  the same expression \eqref{th2eq} remains valid due to the assumption on $\gamma$; namely, $\gamma \in \gf_q$ satisfies $2\gamma\Tr_{q}^{q^n}(c^{q+1}-c^2)\neq 1$ for every $c \in \gf_{q^n}$ with $\Tr_{q}^{q^n}(c^{q+2}-c^3)=0,$

	Taking $\psi_1(x)=\frac{\Tr_{q}^{q^n}\left(x^3-x^{q+2}\right)}{1-2\gamma\Tr_{q}^{q^n}\left(x^{q+1}-x^2\right)},$ $\psi_2(x)=x$, $\varphi_1(x)=\frac{\Tr_{q}^{q^n}\left(f^3(x)-f^{q+2}(x)\right)}{1-2\gamma\Tr_{q}^{q^n}\left(f^{q+1}(x)-f^2(x)\right)}$ and $\varphi_2(x)=f(x)$ in  Lemma \ref{leff-}, then combing $f(x)=x+\gamma \Tr_{q}^{q^n}(x^3-x^{q+2})$ with the expression \eqref{th2eq},  we immediately have $$\varphi_2(x)-\gamma\varphi_1(x)=x.$$
	It follows from Lemma \ref{leff-}  that the compositional inverse of $f(x)$ over $\gf_{q^n}$ is given by 	$$f^{-1}(x)=x-\frac{\gamma\Tr_{q}^{q^n}(x^3-x^{q+2})}{1-2\gamma\Tr_{q}^{q^n}(x^{q+1}-x^{2})}.$$
	This completes the proof. 	\end{proof}

\section*{Declarations}

\begin{itemize}
	\item \textbf{ Conflicts of Interest} There is no conflict of interest.
	\item
The research of Danyao Wu is partially supported by the  National Natural Science Foundation of China (Grant No. 12501006).	The research of Pingzhi Yuan is partially supported by the National Natural Science Foundation of China (Grant Nos. 12571003 and 12171163) and the Guangdong Basic and Applied Basic Research Foundation (Grant No. 2024A1515010589). The research of Zilong He is partially supported by the  National Natural Science Foundation of China (Grant No. 12301013).
\end{itemize}

\bibliography{sn-bibliography1}

\end{document}